\documentclass[10pt,a4paper,twoside]{article}
\usepackage{amsthm,amsfonts,amssymb,amsmath,amscd,bezier,color}
\usepackage{graphicx}
\usepackage{hyperref}
\usepackage{extarrows}
\usepackage{enumerate}

\usepackage[all]{xy}

\newtheorem{prop}[]{Proposition}
\newtheorem{lemma}[prop]{Lemma}
\newtheorem{defi}[prop]{Definition}
\newtheorem{remark}[prop]{Remark}
\newtheorem{theo}[prop]{Theorem}
\newtheorem{coro}[prop]{Corollary}
\newtheorem{exam}[prop]{Example}

\begin{document}

\thispagestyle{empty}

\begin{center}

{\Large\sc Intermediate subgroups of braid groups \\ are not bi-orderable} 

\vspace{8mm}

{\large Raquel M. de A. Cruz}

\vspace{5mm}

\end{center}

\noindent{\bf Abstract:} Let $M$  be the disk or a compact, connected surface without boundary different from the sphere $S^2$ and the real projective plane $\mathbb{R}P^2$, and let $N$ be a compact, connected surface (possibly with boundary). It is known that the pure braid groups $P_n(M)$ of $M$ are bi-orderable, and, for $n\geq 3$, that the full braid groups $B_n(M)$ of $M$ are not bi-orderable. The main purpose of this article is to show that for all $n \geq 3$, any subgroup $H$ of $B_n(N)$ that satisfies $P_n(N) \subsetneq H \subset B_n(N)$ is not bi-orderable.

\vspace{5mm}


\vspace{3mm}



\section{Introduction}\label{section:introduction}

\hspace{4mm} A group $G$ is said to be \textit{left-orderable} if it can be equipped with a strict total ordering $<$ which is left-invariant. More precisely, if $x < y$, then $gx < gy$ for all $x,y,g \in G$. If the ordering $<$ is also right-invariant, then $(G, <)$ is said to be \textit{bi-orderable}. Some examples of bi-orderable groups are free groups \cite{Shimbireva1947}, Thompson's groups \cite{BurilloGM2007, NavasRivas2010}, torsion-free nilpotent groups \cite[Proposition 3.16]{Paris2000}, diagram groups \cite[Theorem 6.1]{GubaSapir2006}, and the fundamental group of some $3$-manifolds \cite{RW1}. Left-orderable groups that cannot be bi-ordered include the groups $\operatorname{Homeo}_{+}(\mathbb{R})$ and $\operatorname{Homeo}_{+}(\mathbb{Q \times Q})$ of order-preserving homeomorphisms \cite{Navas2010}, the universal covering group of the group $SL_2 (\mathbb{R})$ of $2 \times 2$ real matrices with determinant equal to one \cite{Bergman1991}, and the fundamental group of the Klein bottle \cite{RW1}.

Let $M$ be a compact, connected surface, and let $n \geq 1$. The set:
\begin{equation*}
\{(z_1, \dots, z_n); \ z_i \in M \ \text{and} \ z_i \neq z_j, \ \text{for} \ i \neq j\}
\end{equation*}
is called the \textit{nth (ordered) configuration space of} $M$ and is denoted by $F_n(M)$. The map $( \sigma, (z_1, \dots, z_n)) \mapsto (z_{\sigma(1)}, \dots, z_{\sigma(n)})$ defines a free action of $S_n$, the symmetric group on $n$ letters, on $F_n(M)$. The corresponding quotient, which is the orbit space $F_n(M)/S_n$, will be denoted by $C_n(M)$, and is called the \textit{nth unordered configuration space of} $M$. The \textit{pure braid group of $M$ on} $n$ \textit{strands}, denoted by $P_n(M)$, is defined to be the fundamental group of $F_n(M)$, and the \textit{(full) braid group of} $M$ on $n$ \textit{strands}, denoted by $B_n(M)$, is defined to be the fundamental group of $C_n(M)$ \cite{FN, FoxNeuwirth1962, Z}. 

The braid groups $B_n(D)$ of the disk $D$, also known as \textit{Artin braid groups}, are residually finite \cite[Proposition 5]{A1, Baumslag1963, Cohen1989}, Hopfian \cite{Malcev1940} and linear \cite{Bigelow2000, Krammer2000}. They admit solvable word problem \cite{Artin1947} and conjugacy problem \cite{G}, and they are not locally indicable for $n \geq 5$ \cite{RhemtullaRolfsen2002}. Moreover, if $M$ is a compact, connected surface without boundary, the braid groups $B_n(M)$ are torsion-free with the exception of the sphere $S^2$ and the real projective plane $\mathbb{R}P^2$ \cite{FN, FV, VB}. If $M \neq D, S^2$ is such that its fundamental group has trivial center, then the center of $B_n(M)$ is trivial \cite{GonçalvesGuaschi2004, PR}. More properties of surface braid groups may be found in \cite{Birman1974, GM2011, GuaschiJP2015, Hansen1989, PR}.

As proved by Dehornoy \cite{D}, braid groups of the disk are left-orderable. However, it is known that they cannot be bi-ordered, because they do not have the unique root property (see Equation \ref{eq:3}). Since $B_n(D) \subset B_n(M)$, where $D$ is a topological disk in $M$ and $M$ is a compact, connected surface different from $S^2$ and $\mathbb{R}P^2$ \cite{PR}, $B_n(M)$ is not bi-orderable for $n \geq 3$. Garside \cite{G} defined a partial order on the semigroup $B_n(D)^{+}$ of the so-called \textit{positive braids}, i.e. non-trivial braids of the disk that can be written in terms of positive powers of the Artin generators (see Theorem \ref{theo:presentationBn}). Later, Elrifai and Morton \cite{EM} gave a partial ordering of $B_n(D)$.

The pure braid groups of the disk and of compact, connected orientable surfaces with the exception of the sphere $S^2$ are bi-orderable \cite{GM, KR}. However, if $N$ is a non-orientable surface and $n \geq 2$, González-Meneses \cite{GM} exhibited generalized torsion elements in $P_n(N)$, and therefore $P_n(N)$ is not bi-orderable. Moreover, if $M$ is $S^2$ or $\mathbb{R}P^2$, then $P_n(M)$ contains torsion elements \cite{FV, VB}, so it is not left-orderable, and consequently, the braid groups $B_n(S^2)$ and $B_n(\mathbb{R}P^2)$ are not left-orderable.

In this article, we will study the bi-orderability of subgroups $H$ of $B_n(M)$ satisfying $P_n(M) \subsetneq H \subset B_n(M)$, where either $M = D$ or $M$ is a compact, connected surface different from $S^2$ and $\mathbb{R}P^2$. We shall refer to such subgroups as \textit{intermediate subgroups of} $B_n(M)$. We first discuss the case of the disk, and we then use a restriction of the embedding $B_n(D) \hookrightarrow B_n(M)$ \cite{Go, PR}, induced by the inclusion of $D$ in $M$, to extend our results to compact, connected surfaces. More precisely, our main results are as follows.

\begin{theo}\label{theo:teorema1}
Let $n \geq 3$. If $H$ is a bi-orderable subgroup of $B_n(D)$ such that $P_n(D) \subset H \subset B_n(D)$, then $H = P_n(D)$, i.e. $P_n(D)$ is a maximal bi-orderable subgroup of $B_n(D)$.
\end{theo}

\begin{theo}\label{teotoro}
Let $n \geq 3$, and let $M \neq S^2, \mathbb{R}P^2$ be a compact, connected surface. Let $H$ be an intermediate subgroup of $B_n(M)$. Then, $H$ is not bi-orderable.
\end{theo}

Besides the Introduction, this paper consists of two sections. In Section \ref{sec:2}, we give preliminary definitions and we state some well-known results on surface braid groups and orderable groups. The goal of Section~\ref{section:2}, which consists of four subsections, is to investigate the bi-orderability of a certain family of intermediate subgroups of the Artin braid groups, and to prove Theorems \ref{theo:teorema1} and \ref{teotoro}. We finish the paper with some consequences of Theorem~\ref{theo:teorema1}, such as the fact that no intermediate subgroups of the braid group $B_{\infty}(D)$ of the disk on infinitely many strands can be bi-ordered.

\section{Preliminaries}\label{sec:2}

In this section, we recall some results on the braid groups of the disk, such as the classical presentation due to Artin, the description of the center and the structure of the pure braid groups. We also summarize some properties of orderable groups, and we discuss the relation between the unique root property and the existence of generalized torsion elements in a group.

\subsection{Artin braid groups}

\hspace{4mm} Let $M$ be a compact, connected surface, and let $n \geq 1$. The canonical projection $F_n(M) \to C_n(M)$ is a regular $n!$-fold covering map, which leads to the following short exact sequence:
\begin{equation}\label{eq:shorexactsequence}
1 \longrightarrow P_n(M) \longrightarrow B_n(M) \xlongrightarrow[ ]{\pi} S_n \longrightarrow 1,
\end{equation}
where $\pi \colon B_n(M) \to S_n$ associates the permutation corresponding to the endpoints of the strands of a braid. The group homomorphism $\pi$ is called the \textit{permutation homomorphism}.

When $M$ is the disk $D$, Artin \cite{A1, A2} studied geometric and algebraic properties of its braid groups, and provided presentations for both $B_n(D)$ and $P_n(D)$.

\begin{theo}[Artin, \cite{A1, A2}]\label{theo:presentationBn}
For $n \geq 1$, $B_n(D)$ has the following presentation:
\begin{equation}\label{eq:artin}
 \Biggl\langle	\sigma_1, \dots, \sigma_{n-1} \ \Biggl| \ \begin{array}{l} \sigma_i \sigma_j = \sigma_j \sigma_i, \ \text{if} \ |i-j| \geq 2 \\ \sigma_i \sigma_{i+1} \sigma_i = \sigma_{i+1} \sigma_i \sigma_{i+1}, \ \text{for} \ i = 1,2, \dots, n-2 \end{array} \Biggl\rangle.	
\end{equation}
\end{theo}

The generators $\sigma_i$ are called \textit{Artin generators}. They may be interpreted geometrically as a clockwise half-twist of the $i$-th and $(i + 1)$-th strands, as shown in Figure \ref{d5}, where the $i$-th strand is colored in red and the $(i+1)$-th strand is in green.
\begin{figure}[htbp]
\centering
\includegraphics[scale=0.2]{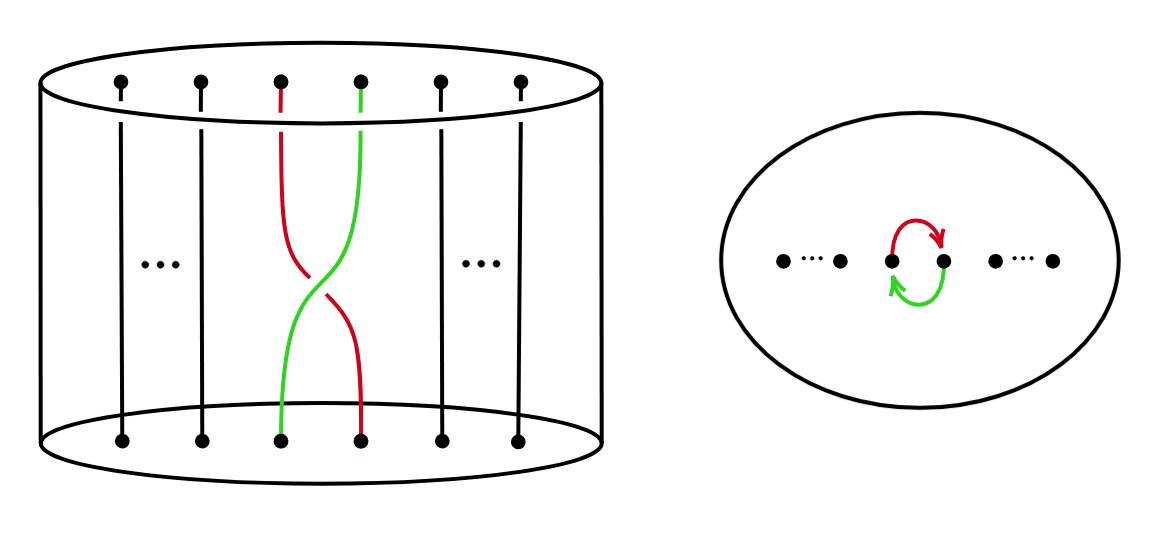}
\caption{The braid $\sigma_i \in B_n(D)$ illustrated geometrically and as an element of $\pi_1(C_n(D))$.}\label{d5}
\end{figure}
\begin{remark}\label{remarkinclusion1}
Let $n \geq 1$. The braid group $B_i(D)$ may be viewed as a subgroup of $B_n(D)$ for all $1 \leq i \leq n$. Indeed, it follows directly from the fact that the homomorphism $\iota \colon B_i(D) \to B_{i+1}(D)$ defined by $\iota (\sigma_k) = \sigma_k$ for all $1 \leq k \leq n-1$ is injective for all $i \geq 1$ \cite[Corollary~1.14]{KasselTuraev2008}.  
\end{remark}

We recall that the center of $B_n(D)$ is infinite cyclic \cite{Chow}. 

\begin{prop}[Chow, \cite{Chow}]\label{rem:propriedadeshalftwist}
For $n \geq 3$, the center of $B_n(D)$ is the cyclic subgroup $\langle \Delta_n^2 \rangle$. The braid $\Delta_n^2$ is called the \textit{full-twist}, and it satisfies $\Delta_n^2 = (\sigma_1 \cdots \sigma_{n-1})^n$. 
\end{prop}

\begin{remark}\label{rem:6}
Let $n \geq 3$ and $2 \leq k \leq n$. The braid $\Delta_k = \sigma_1 (\sigma_2 \sigma_1) \cdots (\sigma_{k-1} \cdots \sigma_1)$ of $B_k(D)$ is called the \textit{half-twist} on $k$ strands, and its square is equal to the full-twist. We recall that:
\begin{equation*}
\pi(\Delta_k) = \begin{cases} (1,k)(2, k-1) \cdots (\frac{k}{2}, \frac{k}{2} +1), & \text{if} \ k \ \text{is even}; \\ (1,k)(2,k-1) \cdots (\frac{k-1}{2}, \frac{k+3}{2}), & \text{if} \ k \ \text{is odd}. \end{cases} 
\end{equation*}
That is:
\begin{equation}\label{eq:set9}
\pi(\Delta_k) = \displaystyle\prod_{l=1}^{\lfloor \frac{k}{2} \rfloor} (l, (k+1)-l),
\end{equation}
a product of disjoint transpositions, where $\lfloor \frac{k}{2} \rfloor$ denotes the greatest integer less than or equal to $\frac{k}{2}$.
\end{remark}

We now give a preliminary result, which will be used in the calculations of Subsection~\ref{subsec:cis}.

\begin{lemma}[Garside, \cite{G}]\label{rem:propriedadeshalftwist1}
Let $1 \leq i \leq k-1$, and let $2 \leq k \leq n$. Then, $\sigma_i \Delta_k = \Delta_k \sigma_{k-i}$.  
\end{lemma}

By (\ref{eq:shorexactsequence}), the pure braid group of the disk is a finite-index subgroup of $B_n(D)$. Moreover, $P_n(D)$ is residually torsion-free nilpotent \cite{FalkRandell1988}, and its center coincides with the center of $B_n(D)$ for all $n \geq 3$ \cite{Birman1974, Chow}. Also, $P_n(D)$ is isomorphic to a semi-direct product of free groups \cite{Artin1947}. 

\begin{theo}[Artin, \cite{A1, A2}]\label{theo:presentationPn}
Let $n \geq 1$. The subgroup $P_n(D)$ is finitely presented, and it is generated by the set $\{A_{i,j}; \ 1 \leq i < j \leq n\}$, where $A_{i,j} = (\sigma_{j-1} \cdots \sigma_{i+1}) \sigma_i^2 (\sigma_{i+1}^{-1} \cdots \sigma_{j-1}^{-1})$.
\end{theo}

Let $n \geq 3$, and $1 \leq i \leq n$. Let
\begin{equation*}
\begin{array}{cccl}
\varphi_i \colon & F_n(D) & \longrightarrow & F_{n-1}(D) \\
& (x_1, \dots, x_n) & \longmapsto & (x_1, \dots, x_{i-1}, x_{i+1}, \dots, x_n)
\end{array}
\end{equation*}
be the projection that omits the $i$-th coordinate of an element $(x_1, \dots, x_n) \in F_n(D)$. Then, the map $\varphi$ is a locally-trivial fibration \cite{FN}, and so it induces an epimorphism $(\varphi_i)_{\#} \colon P_n(D) \to P_{n-1}(D)$, which may be interpreted geometrically as the retraction that forgets the $i$-th strand of a given pure braid. 

\begin{remark}\label{rem.1}
For $n=3$, we have $(\varphi_1)_{\#} (A_{1,3}) = (\varphi_1)_{\#} (A_{2,3}) = A_{1,2} \neq 1 = (\varphi_1)_{\#} (A_{1,2})$, and $(\varphi_2)_{\#} (A_{1,3}) = A_{1,2} \neq 1 = (\varphi_1)_{\#} (A_{2,3})$.
\end{remark}

We now recall the fact that the inclusion of a subsurface of a surface into the surface gives rise to an embedding between their braid groups \cite{PR}. We shall use this result in the case where the subsurface is a disk.

\begin{prop}[Paris and Rolfsen, Proposition 2.2 \cite{PR}]\label{prop2.2}
Let $M$ be a connected surface different from $S^2$ and $\mathbb{R}P^2$, and let $N$ be a connected subsurface of $M$ such that none of the connected components of the closure $\overline{M \setminus N}$ is a disk. Let $m,n \geq 1$, and let $\psi \colon B_n(N) \to B_m(M)$ be the homomorphism induced by the natural embedding $N \hookrightarrow M$. Then $\psi$ is injective.
\end{prop}

\begin{remark}\label{rem.2}
Taking $m=n$ and $N$ to be a topological disk $D$ in Proposition \ref{prop2.2}, we obtain $B_n(D) \cong \psi(B_n(D)) \subset B_n(M)$ for all $n \geq 1$ and $M \neq S^2, \mathbb{R}P^2$ a connected surface. Hence, $B_n(D)$ may be viewed as a subgroup of $B_n(M)$. Moreover, by the proof of Proposition \ref{prop2.2} \cite[Proposition 2.2, p. 424]{PR}, we have $P_n(D) \cong \psi(P_n(D)) \subset P_n(M)$, which was observed earlier by Goldberg \cite[Theorem 1]{Go}.
\end{remark}

\subsection{Ordered groups}

\hspace{4mm} In the literature, one may find authors who study left-orderings (eg. \cite{BRW, CR, Navas2010}) and others who study right-orderings (eg. \cite{Kopytov2013, Passman, RhemtullaRolfsen2002}). In this paper, we choose the convention of studying left-orderable groups. Recall that a left-orderable group $(G, <)$ is clearly right-orderable with the ordering $\prec$ defined by $ x \prec y \ \Leftrightarrow \ y^{-1} < x^{-1}$.

A well-known characterization of ordered groups, first established by Levi \cite{Levi1} in the context of bi-invariant orderings, is given in terms of the existence of a \textit{positive cone}. 

\begin{prop}[Levi, \cite{Levi1}]\label{prop:positivecone}
A group $G$ is left-orderable if and only if there exists a set $P \subset G$, known as the positive cone of the ordering, such that $P P \subset P$ and $G = P \mathop{\dot{\bigcup}} \{1\} \mathop{\dot{\bigcup}} P^{-1}$, where the union is disjoint. Moreover, $G$ is bi-orderable if and only if $P$ is a normal set, namely $P^g = gPg^{-1} \subset P$ for all $g \in G$.
\end{prop}

Left-orderable groups are torsion-free. Furthermore, bi-orderable groups have the \textit{unique root property} \cite[Lemma 2.1.4]{Glass1999}, and they do not have \textit{generalized torsion elements} \cite[Lemma 2.3]{MotegiTeragaito2017}, as we recall in Proposition \ref{lema:raizesunicas} $(i)$ and $(ii)$ respectively. We also recall that a non-trivial group element is said to be a \textit{generalized torsion element} if there exists a non-empty finite product of its conjugates that is equal to the identity.   

\begin{prop}\label{lema:raizesunicas}
Let $G$ be a bi-orderable group. Then, the following hold:

\begin{enumerate}[(i)]
\item Let $g,h \in G$. If there exists $n \geq 1$ such that $g^n = h^n$, then $g=h$.

\item If $g \in G \setminus \{1\}$, then the product $(x_1 g x_1^{-1})(x_2 g x_2^{-1}) \cdots (x_k g x_k^{-1})$ is non-trivial for all $x_1, \dots, x_k \in G$ and $k \geq 1$.
\end{enumerate} 
\end{prop}

By the following proposition, a group with non-unique roots also has generalized torsion elements \cite{NR}, and so $(i)$ implies $(ii)$ in Proposition \ref{lema:raizesunicas}.  

\begin{prop}[Naylor and Rolfsen, \cite{NR}]\label{prop:gt}
Let $p, q \in \mathbb{Z}$ be such that $|p|, |q| > 1$. The group
\begin{equation*}
G_{p,q} = \langle x,y; \ x^p = y^q \rangle
\end{equation*}
contains a generalized torsion element, namely the commutator $\left[x,y\right] = xyx^{-1}y^{-1}$.
\end{prop}

By Proposition \ref{prop:gt}, if there exist $x,y \in G$ such that $x^p = y^q$, for $|p|, |q| > 1$, and $x$ and $y$ do not commute, then $[x,y] \in G_{p,q} \subset G$ is a generalized torsion element of $G$, and therefore $G$ is not bi-orderable. We will use this result in Remark \ref{rem:1} to exhibit generalized torsion elements of intermediate subgroups of the braid groups of the disk.

\begin{exam}
With the notation of Proposition \ref{prop:gt}, let $p=q=2$, and let us write $x^y = yxy^{-1}$ for all $x,y \in G$. To obtain a generalized torsion relation satisfied by the element $[x,y]$ of the group $G_{2,2} = \langle x,y; \ x^2 = y^2 \rangle$, observe that:
\begin{equation*}
\left[x^2,y^2\right]  =  \left[x, y^2 \right]^x \left[x,y^2\right] =  \left( \left[x, y \right] \left[x, y \right]^y \right)^x \left[x,y\right] \left[x, y \right]^y = \left[x, y \right]^x \left[x, y \right]^{xy} \left[x,y\right] \left[x, y \right]^y. 
\end{equation*}
Then, since $x^2 = y^2$, we have:
\begin{equation}\label{eq:out17}
\left[x, y \right]^x \left[x, y \right]^{xy} \left[x,y\right] \left[x, y \right]^y = 1.
\end{equation}
\end{exam}

More details and properties of ordered groups may be found in \cite{CR, MuraRhemtulla1977, Passman}. We end this subsection with some comments on \textit{partial orderings}, for which we refer the reader to Fuchs's \cite{Fuchs2014} and Glass's \cite{Glass1999} books. A group $G$ is said to be \textit{partially left-orderable} if there exists a partial strict ordering $<$ on its elements that is left-invariant. If such an ordering is also right-invariant, then $(G, <)$ is said to be \textit{partially bi-orderable}. As in Proposition~\ref{prop:positivecone}, a group $G$ is partially left-orderable (resp. partially bi-orderable) if and only if there exists a semigroup $Q$ such that $Q \cap Q^{-1} = \emptyset$ (resp. and $Q$ is a normal set). In this case, the semigroup $Q$ is called a \textit{partial positive cone} of the group $G$. An example of a partially left-orderable group that cannot be left-ordered is the multiplicative group of all non-zero principal ideals in a commutative integral domain with unity, and an example of a partially bi-ordered group that is not bi-orderable is the additive group of all polynomial functions with real coefficients whose domain is the closed interval $[0,1]$ \cite[Section II.3]{Fuchs2014}.

\section{Bi-orderability of intermediate subgroups}\label{section:2}

\hspace{4mm} This section consists of four subsections. Although the braid groups of the disk are not bi-orderable, in Subsection \ref{subsec-pb}, we show that they are partially bi-ordered. In Subsection~\ref{subsec:cis}, we investigate some intermediate subgroups of $B_n(D)$. In Subsection \ref{section:3}, by studying the permutation of non-pure braids, we prove Theorem \ref{theo:teorema1}. Finally, in Subsection \ref{subsec:consequences}, we state some consequences of Theorem \ref{theo:teorema1} and give the proof of Theorem \ref{teotoro}. 

\subsection{Partial bi-ordering}\label{subsec-pb}

\hspace{4mm} In this section,  we prove that intermediate subgroups of $B_n(D)$ can be partially bi-ordered.

\begin{prop}\label{prop:ordemparcial}
Let $G$ be a group, and let $f \colon G \to N$ be a group homomorphism. If $\ \operatorname{Im}(f)$ is non-trivial and $N$ is partially bi-orderable, then $G$ is partially bi-orderable. 
\end{prop}

\begin{proof}
Let $Q(N)$ be the partial positive cone of $N$. Then the set $Q = f^{-1}(Q(N))$ is a partial positive cone of $G$.
\end{proof}

\begin{coro}
Intermediate subgroups of the braid groups of the disk are partially bi-orderable. In particular, $B_n(D)$ is partially bi-orderable.
\end{coro}

\begin{proof}
Let $H$ be an intermediate subgroup of $B_n(D)$. Consider the Abelianization homomorphism $\operatorname{exp} \colon B_n(D) \to \mathbb{Z}$ given by $\operatorname{exp}(\sigma_i) = 1$ for all $1 \leq i \leq n-1$. Since  $2\mathbb{Z} = \operatorname{exp}(P_n(D)) \subset \operatorname{exp}(H)$, the image $\operatorname{Im}(\left.\exp\right|_H)$ of the restriction of $\exp$ to $H$ is non-trivial, and so the result follows from Proposition \ref{prop:ordemparcial} and the fact that $\mathbb{Z}$ is clearly partially bi-orderable with the usual ordering.
\end{proof}

\subsection{Non-bi-orderability of certain intermediate subgroups}\label{subsec:cis}

\hspace{4mm} In this subsection, we show that certain intermediate subgroups of braid groups of the disk are not bi-orderable. 

\begin{defi}\label{defi:1}
Let $n \geq 2$. Given $\beta \in B_n(D) \setminus P_n(D)$, let $H_{\beta} = \langle P_n(D), \beta \rangle$, where $\langle P_n(D), \beta \rangle$ is the subgroup generated by $P_n(D)$ and $\beta$.
\end{defi}

We first study the family of intermediate subgroups described in Definition \ref{defi:1}. Note that $B_1(D)$ is trivial, and $B_2(D)$ is isomorphic to $\mathbb{Z}$. So from now on, we suppose that $n \geq 3$. Our goal is to show that the subgroup $H_{\beta}$ is not bi-orderable for every braid $\beta \in B_n (D) \setminus P_n (D)$, which will imply that no intermediate subgroup of the braid groups of the disk can be bi-ordered.

\begin{lemma}\label{lemma:subgrupo1}
Let $f \colon G \to H$ be a group homomorphism. Then $\langle \operatorname{Ker}f, \beta \rangle = f^{-1} ( \langle f(\beta) \rangle )$ for all $\beta \in G$.
\end{lemma}

\begin{proof}
First, note that $f(\langle \operatorname{Ker}f, \beta \rangle) = \langle f(\beta) \rangle$, hence $\langle \operatorname{Ker}(f), \beta \rangle \subset f^{-1}(\langle f(\beta) \rangle)$. Conversely, let $y~\in~f^{-1} ( \langle f(\beta) \rangle)$. Thus, $f (y) = f(\beta)^{\epsilon}$ for some $\epsilon \in \mathbb{Z}$. Then, we have $f ( y\beta^{-\epsilon}) = 1$, and so there exists $\alpha \in \operatorname{Ker}f$ such that $y\beta^{-\epsilon} = \alpha$. Hence $y = \alpha \beta^{\epsilon} \in \langle \operatorname{Ker}f, \beta \rangle$, and we conclude that $\langle \operatorname{Ker}f, \beta \rangle = f^{-1} ( \langle f(\beta) \rangle )$.
\end{proof}

Consider the short exact sequence in (\ref{eq:shorexactsequence}). Given $\beta \in B_n(D)$, by Lemma \ref{lemma:subgrupo1}, we have $\pi^{-1} ( \langle \pi(\beta) \rangle) = \langle \operatorname{Ker}\pi, \beta \rangle = H_{\beta}$. Thus if $\gamma \in B_n(D)$ is such that $\pi(\gamma) = \pi(\beta)$, then $H_{\beta} = H_{\gamma}$, and so to decide whether $H_{\beta}$ is bi-orderable, it suffices to study $H_{\gamma}$ for some fixed braid $\gamma$. Moreover, we shall prove that the bi-orderability of $H_{\beta}$ depends only on the cycle type of the permutation of the braid $\beta$.

\begin{lemma}\label{lemma:conjugatedsubgroups}
Let $H$ and $N$ be subgroups of a group $G$. Suppose that there exists $g \in G$ such that $H^g = N$. Then, $H$ is bi-orderable if and only if $N$ is bi-orderable.
\end{lemma}

\begin{proof}
Suppose that $H$ is bi-orderable, and let $P(H)$ denote its positive cone. The semigroup $P = P(H)^g = \{ ghg^{-1}; \ h \in P(H) \}$ is a positive cone for $N$ and it is invariant under conjugation by elements of $N$. Therefore, $N$ is bi-orderable. Moreover, $H^g = N$ if and only if $H = N^{g^{-1}}$, and this concludes the proof.   
\end{proof}

\begin{prop}\label{prop:tipodepermutacao}
Let $\beta, \gamma \in B_n(D) \setminus P_n(D)$ be such that the permutations $\pi(\beta)$ and $\pi(\gamma)$ have the same cycle type. Then, the subgroups $H_{\beta}$ and $H_{\gamma}$ are conjugate. Moreover, $H_{\beta}$ is bi-orderable if and only if $H_{\gamma}$ is bi-orderable.  
\end{prop}

\begin{proof}
Since $\pi(\beta)$ and $\pi(\gamma)$ have the same cycle type in $S_n$, there exists $\tau \in S_n$ such that $\tau \pi(\beta) \tau^{-1} = \pi(\gamma)$. The map $\pi$ is an epimorphism, so there exists $\alpha \in B_n(D)$ such that $\pi(\alpha) = \tau$. Thus, $\pi(\alpha \beta \alpha^{-1}) = \pi(\gamma)$, and we obtain:
\begin{equation*}
H_{\gamma} = \langle P_n(D), \gamma \rangle = \langle P_n(D), \alpha \beta \alpha^{-1} \rangle = \langle P_n(D), \beta \rangle^{\alpha} = H_{\beta}^{\alpha},  
\end{equation*}
because $P_n(D)$ is a normal subgroup of $B_n(D)$. Since $H_{\beta}$ and $H_{\gamma}$ are conjugate, by Lemma~\ref{lemma:conjugatedsubgroups}, $H_{\beta}$ is bi-orderable if and only if $H_{\gamma}$ is bi-orderable.
\end{proof}

By Proposition \ref{prop:tipodepermutacao}, to decide whether the subgroup $H_{\beta}$ is bi-orderable, where $\beta \in B_n(D)$, it suffices to analyze the bi-orderability of $H_{\gamma}$ for a braid $\gamma \in B_n(D)$ whose permutation has the same cycle type as $\pi(\beta)$. The rest of this section is devoted to studying $H_{\beta}$, where $\pi(\beta)$ is a transposition, a product of disjoint transpositions or a cycle of length greater than $2$. We start with the case where $\pi(\beta)$ is a transposition.

\begin{prop}\label{prop:PnSigma}
Let $n \geq 3$, and let $\beta \in B_n(D)$ be such that $\pi(\beta)$ is a transposition. The subgroup $H_{\beta}$ of $B_n(D)$ is not bi-orderable.
\end{prop}

\begin{proof}
By Proposition \ref{prop:tipodepermutacao}, it suffices to show that $H_{\sigma_1}$ is not bi-orderable. Using the Artin relations (\ref{eq:artin}), we have:
\begin{equation}\label{eq:a}
(\sigma_1\sigma_2^2)^2 = \sigma_1 \sigma_2(\sigma_2 \sigma_1 \sigma_2) \sigma_2 = (\sigma_1 \sigma_2 \sigma_1) (\sigma_2 \sigma_1 \sigma_2) = \sigma_2 (\sigma_1 \sigma_2 \sigma_1) \sigma_2 \sigma_1 = (\sigma_2^2 \sigma_1)^2. 
\end{equation}
Note that $\sigma_1 \sigma_2^2$ and $\sigma_2^2 \sigma_1$ belong to $H_{\sigma_1}$. If $H_{\sigma_1}$ is bi-orderable, by Proposition \ref{lema:raizesunicas}, we have:
\begin{equation*}
\sigma_1A_{2,3} = \sigma_1 \sigma_2^2 = \sigma_2^2 \sigma_1 = A_{2,3} \sigma_1.
\end{equation*}
However, $A_{2,3} \neq \sigma_1^{-1}A_{2,3} \sigma_1 = A_{1,3}$, by Remark \ref{rem.1}. Hence, $H_{\sigma_1}$ does not have the unique root property, and so, by Proposition \ref{lema:raizesunicas}, it is not bi-orderable. \end{proof}

We now give a preliminary result, which we use to show that $H_{\beta}$ is not bi-orderable in the case that $\pi(\beta)$ is a product of disjoint transpositions.

\begin{lemma}\label{lemma:comutatividade1}
Let $\ 3 \leq i \leq n$. In $B_i(D) \subset B_n(D)$, the braids $A_{1,2}$ and $\Delta_i$ do not commute.
\end{lemma}

\begin{proof}
Suppose that there exists $k \geq 3$ such that $A_{1,2}$ commutes with $\Delta_k$. By Lemma~\ref{rem:propriedadeshalftwist1}, we have  $\sigma_1 \Delta_k = \Delta_k \sigma_{k-1}$. Then: 
\begin{equation*}
A_{1,2} = \sigma_1^2 = \Delta_k^{-1} \sigma_1^2 \Delta_k = (\Delta_k^{-1} \sigma_1 \Delta_k)^2 =  \sigma_{k-1}^2 = A_{k-1, k}, 
\end{equation*}
which holds if and only if $k=2$. Therefore, $A_{1,2}$ and $\Delta_i$ do not commute if $3 \leq i \leq n-1$.
\end{proof}

\begin{prop}\label{prop:halftwist}
Let $3 \leq i \leq n$. Then, the subgroup $H_{\Delta_i} = \langle P_n(D), \Delta_i \rangle$ is not bi-orderable.
\end{prop}

\begin{proof}
Let $i \in \{3, \dots, n\}$. Since the center of $B_i(D)$ is $\langle \Delta_i^2 \rangle$ by Proposition \ref{rem:propriedadeshalftwist}, we have:
\begin{equation}\label{eq:5}
(A_{1,2} \Delta_i A_{1,2}^{-1})^2 = A_{1,2} \Delta_i^2 A_{1,2}^{-1} = \Delta_i^2 = (\Delta_i)^2.
\end{equation}
So, if $H_{\Delta_i}$ is bi-orderable, then $A_{1,2} \Delta_i A_{1,2}^{-1} = \Delta_i$, which contradicts Lemma \ref{lemma:comutatividade1}.
\end{proof}

Finally, we study the intermediate subgroup $H_{\beta}$ for a braid $\beta \in B_n(D)$ whose permutation is a cycle of length strictly greater than 2.

\begin{lemma}\label{lemma:raizdohalftwist}
Let $n \geq 4$. In $B_{i+1}(D) \subset B_n(D)$, with $2 \leq i \leq n-1$, the braids $A_{1,2}$ and $\sigma_1 \sigma_2 \cdots \sigma_{i}$ do not commute.
\end{lemma}

\begin{proof}
Suppose on the contrary that there exists $i \in \{2, \dots, n-1\}$ such that $A_{1,2}$ commutes with $\sigma_1 \sigma_2 \cdots \sigma_{i}$. Using the relation $\sigma_j \sigma_k = \sigma_k \sigma_j$, whenever $|j - k| \geq 2$, we have:
\begin{align*}
1 =[A_{1,2}, \sigma_1 \sigma_2 \sigma_3 \cdots \sigma_i] & = 
\sigma_1^2 \sigma_1 \sigma_2 (\sigma_3 \cdots \sigma_i) \sigma_1^{-2} (\sigma_3 \cdots \sigma_i)^{-1} \sigma_2^{-1} \sigma_1^{-1} \\ 
& =  \sigma_1^2 \sigma_1 \sigma_2 \sigma_1^{-2}  (\sigma_3 \cdots \sigma_i) (\sigma_3 \cdots \sigma_i)^{-1} \sigma_2^{-1} \sigma_1^{-1} \\ 
& = \sigma_1 A_{1,2} A_{1,3}^{-1} \sigma_1^{-1}. 
\end{align*}
Thus $A_{1,2} = A_{1,3}$, which is absurd by Remark \ref{rem.1}, and this proves the lemma.
\end{proof}

\begin{prop}\label{prop:Pn2Sigma}
Let $n \geq 3$, and let $2 \leq i \leq n-1$. Let $\beta \in B_n(D)$ be such that $\pi(\beta)$ is a cycle of length $i+1$. Then, $H_{\beta}$ is not bi-orderable.
\end{prop}

\begin{proof}
Since $\pi(\sigma_1 \cdots \sigma_{i})$ is an $(i+1)$-cycle, by Proposition \ref{prop:tipodepermutacao}, it suffices to prove the statement for $H_{\sigma_1 \cdots \sigma_{i}} = \langle P_n, \sigma_1 \cdots \sigma_{i} \rangle$. First let $n=3$, so $i=2$, and consider the subgroup $H_{\sigma_1 \sigma_2}$. From Lemma \ref{lemma:subgrupo1}, the braid $\sigma_{2} \sigma_1$ belongs to $\pi^{-1} ( \langle \pi(\sigma_1 \sigma_{2} \rangle ) = H_{\sigma_1 \sigma_{2}}$, because $\pi (\sigma_{2} \sigma_1) = (1,2, 3) = (1, 3, 2)^2 = \pi(\sigma_1 \sigma_{2})^2$. Note that:
\begin{equation}\label{eq:3}
(\sigma_1 \sigma_{2})^3 = (\sigma_1 \sigma_{2} \sigma_1) (\sigma_{2} \sigma_1 \sigma_{2}) = (\sigma_{2} \sigma_1 \sigma_{2}) (\sigma_1 \sigma_{2} \sigma_1) = (\sigma_{2} \sigma_1)^3.
\end{equation}
If $H_{\sigma_1 \sigma_{2}}$ is bi-orderable, Equation (\ref{eq:3}) implies that $\sigma_1 \sigma_{2} = \sigma_{2} \sigma_1$, which yields a contradiction, since $\pi(\sigma_1 \sigma_{2}) = (1, 3, 2) \neq (1, 2, 3) = \pi(\sigma_{2} \sigma_1)$. 

Now suppose that $n \geq 4$. In $B_{i+1} (D) \subset B_n(D)$, since $\Delta_{i+1}^2 = (\sigma_1 \cdots \sigma_{i})^{i+1}$ and this element is central in $B_{i+1}$, we have:
\begin{equation*}
\left(A_{1,2} (\sigma_1 \cdots \sigma_i) A_{1,2}^{-1} \right)^{i+1} = (\sigma_1 \cdots \sigma_i)^{i+1}.
\end{equation*}
If $H_{\sigma_1 \cdots \sigma_{i}}$ is bi-orderable, by Proposition \ref{lema:raizesunicas}, it follows that $A_{1,2} (\sigma_1 \cdots \sigma_i) A_{1,2}^{-1} = \sigma_1 \cdots \sigma_i$. So $A_{1,2}$ commutes with $\sigma_1 \cdots \sigma_i$, which contradicts Lemma~\ref{lemma:raizdohalftwist}. Hence, the subgroup $H_{\sigma_1 \cdots \sigma_{i}}$ is not bi-orderable.
\end{proof}

\subsection{Proof of Theorem 1}\label{section:3}

\hspace{4mm} We now prove Theorem \ref{theo:teorema1}.

\begin{proof}
Let us show that $H_{\beta} \subset H$ cannot be bi-ordered for all $\beta \in H \setminus P_n (D)$, which will imply that $H$ is not bi-orderable.

First, let $n=3$. Since any non-trivial permutation of $S_3$ is either a transposition or a $3$-cycle, by Proposition \ref{prop:tipodepermutacao}, the intermediate subgroup $H_{\beta}$ of $B_3$ is either conjugate to $H_{\sigma_1}$ or $H_{\sigma_1 \sigma_2}$. From Propositions \ref{prop:PnSigma} and \ref{prop:Pn2Sigma}, we see that $H_{\beta}$ is not bi-orderable. Now let $n \geq 4$, and assume that $P_n (D) \subsetneq H$. 

Now, let $\beta \in H \setminus P_n (D)$. We shall divide the proof into two cases:

\begin{enumerate}[(a)]
\item $\pi(\beta)$ is a product of disjoint transpositions;
\item $\pi(\beta)$ is a product of disjoint cycles, one of which is of length greater than or equal to $3$.
\end{enumerate}

For case $(a)$, we may write $\pi(\beta) = \displaystyle\prod_{l=1}^m (i_l, j_l)$, a product of disjoint transpositions, where $1 \leq m \leq \lfloor \frac{n}{2} \rfloor$ and $i_1, j_1, \dots, i_m, j_m$ are pairwise distinct. For $2 \leq i \leq m$, recall that $\Delta_i = \sigma_1 (\sigma_2 \sigma_1) \cdots (\sigma_{i-1} \cdots \sigma_2 \sigma_1)$. Then, by Equation (\ref{eq:set9}), we have:
\begin{equation*}
\pi(\Delta_{2m}) = \displaystyle\prod_{k=1}^{m} (k,  (i+1)-k), 
\end{equation*}
a product of $m$ disjoint transpositions, therefore $\Delta_{2m}$ and $\beta$ have the same cycle type, and so $H_{\beta}$ and $H_{\Delta_{2m}}$ are conjugate by Proposition \ref{prop:tipodepermutacao}. If $m=1$, $H_{\Delta_{2}} = H_{\sigma_1}$, and Proposition~\ref{prop:PnSigma} shows that $H_{\Delta_{2}}$ is not bi-orderable, and if $m \geq 2$, $H_{\Delta_{2m}}$ is not bi-orderable by Proposition~\ref{prop:halftwist}.

For case $(b)$, we may write $\pi(\beta) = i_1 \cdots i_l$, a product of disjoint cycles $i_k$ of length $j_k$, with $1 \leq k \leq l$, $1 \leq l \leq \lfloor \frac{n-1}{2} \rfloor$, $2 \leq j_k \leq n$, $j_p \geq 3$ for some $p \in \{1, \dots, l\}$, and $\displaystyle\sum_{k=1}^l j_k \leq n$. Since the cycles are disjoint, up to renumbering $i_1, \dots, i_l$, we may suppose that $p = 1$, and so $j_1 \geq 3$. Now, define $\alpha \in B_n (D)$ by:
\begin{equation}\label{eq:10}
\alpha = \displaystyle\prod_{k=1}^l \alpha_k, \ \text{with} \ \alpha_k = \sigma_{\sum_{m=1}^{k-1} j_m + 1} \cdots \sigma_{\sum_{m=1}^{k} j_m - 1},
\end{equation}
that is:
\begin{equation*}
\alpha = (\sigma_1 \cdots \sigma_{j_1 - 1})(\sigma_{j_1 + 1} \cdots \sigma_{j_1 + j_2 - 1}) \cdots (\sigma_{j_1 + \cdots + j_{l-1} +1} \cdots \sigma_{j_1 + \cdots + j_{l} -1}).
\end{equation*} 
Note that: 
\begin{equation*}
\pi(\alpha) = \displaystyle\prod_{k=1}^l \pi(\alpha_k) = \displaystyle\prod_{k=1}^l \left( \displaystyle\sum_{m=1}^k j_m , \displaystyle\sum_{m=1}^k j_m -1 , \cdots , \displaystyle\sum_{m=1}^{k-1} j_m +1 \right).
\end{equation*}
Then, $\pi(\alpha)$ is a product of disjoint cycles, where the cycle $\pi(\alpha_k)$ is of length $j_k$, and so the permutations of $\alpha$ and $\beta$ have the same cycle type. By Proposition \ref{prop:tipodepermutacao}, $H_{\beta}$ is  bi-orderable if and only if $H_{\alpha}$ is. We shall show that the unique root property does not hold in $H_{\alpha}$, and consequently, by Proposition \ref{lema:raizesunicas}, this subgroup cannot be bi-ordered. By the Artin relations (\ref{eq:artin}) and Equation~(\ref{eq:10}), the braids $\alpha_i$ and $\alpha_j$ commute pairwise for every $i,j \in \{1, \dots, l\}$. Note also that $A_{1,2} = \sigma_1^2$ commutes with $\alpha_i$ for every $i \in \{2, \dots, l\}$ because $j_1 \geq 3$. Since $\alpha_1^{j_1} = (\sigma_1 \cdots \sigma_{j_1 - 1})^{j_1} = \Delta_{j_1}^2$ is a central element of $B_{j_1} (D) \subset B_n (D)$, by Proposition \ref{rem:propriedadeshalftwist}, we have: 
\begin{align}\label{eq:13}
(A_{1,2} \alpha A_{1,2}^{-1})^{j_1} & =  A_{1,2} \alpha^{j_1} A_{1,2}^{-1} 
 =  A_{1,2} \alpha_1^{j_1} \alpha_2^{j_1} \cdots \alpha_l^{j_1} A_{1,2}^{-1} 
 =  \alpha_1^{j_1}  A_{1,2} (\alpha_2^{j_1}  \cdots \alpha_l^{j_1})A_{1,2}^{-1} \\ \notag
& =  \alpha_1^{j_1}  \alpha_2^{j_1}  \cdots \alpha_l^{j_1} A_{1,2}A_{1,2}^{-1} 
 =  \alpha^{j_1}. \notag
\end{align}
If the subgroup $H_{\alpha}$ is bi-orderable, by Proposition \ref{lema:raizesunicas}, we have $A_{1,2} \alpha A_{1,2}^{-1} = \alpha$. However:
\begin{align*}
[A_{1,2}, \alpha] & = A_{1,2}\alpha A_{1,2}^{-1} \alpha^{-1} 
 = \sigma_1^2 \alpha_1 (\alpha_2 \cdots \alpha_l)\sigma_1^{-2} (\alpha_2 \cdots \alpha_l)^{-1} \alpha_1^{-1} \\ 
& = \sigma_1^2 \alpha_1 (\alpha_2 \cdots \alpha_l)(\alpha_2 \cdots \alpha_l)^{-1} \sigma_1^{-2}  \alpha_1^{-1} 
 = \sigma_1^2 \alpha_1 \sigma_1^{-2}  \alpha_1^{-1} \\ 
& = \sigma_1^2 \sigma_1 \sigma_2 (\sigma_3 \cdots \sigma_{j_1 - 1}) \sigma_1^{-2} (\sigma_3 \cdots \sigma_{j_1 - 1})^{-1}\sigma_2^{-1} \sigma_1^{-1}  \\ 
& = \sigma_1^2 \sigma_1 \sigma_2  (\sigma_3 \cdots \sigma_{j_1 - 1}) (\sigma_3 \cdots \sigma_{j_1 - 1})^{-1}  \sigma_1^{-2} \sigma_2^{-1} \sigma_1^{-1}  \\ 
& = \sigma_1^2 \sigma_1 (\sigma_1^{-1} \sigma_2^{-2} \sigma_1) \sigma_1^{-1}   = A_{1,2}A_{2,3}^{-1}, 
\end{align*}
and so $A_{1,2} = A_{2,3}$, which contradicts Remark \ref{rem.1}. Therefore, $H_{\alpha}$ cannot be bi-ordered, and so neither can $H_{\beta}$.

From cases $(a)$ and $(b)$, we conclude that $H_{\beta}$ is not bi-orderable for all possible cycle types of the permutation $\pi(\beta) \in S_n$. Hence, if $H$ is an intermediate subgroup of $B_n(D)$, then $H$ contains $H_{\beta}$ for some $\beta \in B_n(D) \setminus P_n(D)$, and so $H$ is not bi-orderable.
\end{proof}

\subsection{Consequences and proof of Theorem 2}\label{subsec:consequences}

\hspace{4mm} This subsection is devoted to giving some consequences of Theorem \ref{theo:teorema1}. We recall that the braid group $B_{\infty}(D)$ of the disk on an infinite number of strands is composed of braids with an infinite number of strands but finitely many crossings. More precisely, $B_{\infty}(D)$ is the direct limit of $B_n(D)$ under the injective homomorphism $\iota \colon B_n(D) \to B_{n+1}(D)$ given by $\iota(\sigma_i) = \sigma_i$ for all $1 \leq i \leq n-1$ (see \cite{D1, F1} and Remark \ref{remarkinclusion1}). There exists an epimorphism $\pi_{\infty} \colon B_{\infty} \to S_n$ such that the restriction $\left.\pi_{\infty}\right|_{B_{n(D)}}$ is the permutation homomorphism $\pi \colon B_n(D) \to S_n$. The kernel of the map $\pi_{\infty}$ is defined to be $P_{\infty}(D)$, which is the direct limit of $P_n(D)$ under the restriction $\left.\iota\right|_{P_n(D)}$. We now show that intermediate subgroups of $B_{\infty}(D)$ are not bi-orderable.

\begin{coro}\label{cor:Binfinito}
Let $H$ be an intermediate subgroup of $B_{\infty}$(D), i.e. $P_{\infty}(D) \subsetneq H \subset B_{\infty}(D)$. Then, $H$ is not bi-orderable.
\end{coro}

\begin{proof}
Let $\beta \in H \setminus P_{\infty}(D)$. Thus, there exist $m \in \mathbb{N}$ and $\beta' \in P_m(D)$ such that $m$ is the least positive integer that satisfies $i(\beta') = \beta$, where $i \colon P_m (D) \to P_{\infty} (D)$ is the natural inclusion. By Theorem~\ref{theo:teorema1}, the subgroup $\langle P_m(D), \beta \rangle$ is not bi-orderable. Identifying this subgroup with its image under $i$, we have $\langle P_m(D), \beta' \rangle \subset \langle P_{\infty}(D), \beta \rangle \subset H $, and so $H$ is not bi-orderable.
\end{proof}

We now prove Theorem \ref{teotoro}, which is an extension of Theorem \ref{theo:teorema1} to intermediate subgroups of compact, connected surfaces different from the sphere and the real projective plane.

\begin{proof}
Since $P_n(M) \subsetneq H$, there exists $\beta \in H \setminus P_n(M)$. Let $\alpha \in B_n(D) \subset B_n(M)$ be such that $\pi(\alpha) = \pi(\beta)$, where $\pi \colon B_n(M) \to S_n$ is the permutation homomorphism. Hence, $\alpha \beta^{-1} \in P_n(M)$. So there exists $\gamma \in P_n(M)$ such that $\alpha = \gamma \beta$ and, consequently, 
\begin{equation*}
\langle P_n(M), \alpha \rangle = \langle P_n(M), \gamma\beta \rangle = \langle P_n(M), \beta \rangle \subset H.
\end{equation*}
Thus, by Remark \ref{rem.2}, we have $H \supset \langle P_n(M), \alpha \rangle \supset \langle P_n(D), \alpha \rangle $. Since $\langle P_n(D), \alpha \rangle$ is not bi-orderable, by Theorem \ref{theo:teorema1}, it follows that $H$ is not bi-orderable either. \end{proof}

If $M$ is a compact, connected, non-orientable surface without boundary, recall that the pure braid groups $P_n(M)$ have generalized torsion for $n \geq 2$ \cite[Theorem 1.7]{GM}, and therefore are not bi-orderable. This immediately implies that no intermediate subgroup of $B_n(M)$ can be bi-orderable for such a surface $M$.

\begin{remark}\label{rem:1}
Let  $M$ be as in Theorem \ref{teotoro}, and let $n\geq 3$. Every intermediate subgroup of $B_n(M)$ has generalized torsion in its commutator subgroup. To see this, using Proposition~\ref{prop:gt} and the fact that $\langle P_n(D), \beta \rangle \subset \langle P_n(M), \beta \rangle$ for all $\beta \in B_n(D)$ (see Remark~\ref{rem.2}), it suffices to show that $H_{\beta}$ does not have the unique root property for certain braids $\beta \in B_n(D)$, which we now describe. For $n=3$, by Equations (\ref{eq:a})
and (\ref{eq:3}), the commutators $\left[\sigma_1A_{2,3}, A_{2,3} \sigma_1 \right]$ and $\left[\sigma_1\sigma_2, \sigma_2 \sigma_1 \right]$ are generalized torsion elements of $H_{\sigma_1}$ and $H_{\sigma_1\sigma_2}$ respectively. For $n\geq 4$, let us first  consider  the case where the  cycle decomposition of  $\pi(\beta)$  has a cycle of length greater than or equal  to $3$. With the notation defined in the proof of Theorem~\ref{theo:teorema1}, the commutator $\left[A_{1,2} \alpha A_{1,2}^{-1}, \alpha\right]$ is a generalized torsion element of $H_{\alpha}$,  by Equation (\ref{eq:13}). Now, suppose that    $\pi(\beta)$ is a product of disjoint transpositions. 
Then  $\left[A_{1,2} \Delta_i A_{1,2}^{-1}, \Delta_i\right]$ is a generalized torsion element of $H_{\Delta_i}$,
by Equation (\ref{eq:5}), for $i\geq 3$. Moreover, the elements $[\sigma_1A_{2,3}, A_{2,3}\sigma_1], [\sigma_1 \sigma_2, \sigma_2 \sigma_1]$ and $\left[A_{1,2} \Delta_i A_{1,2}^{-1}, \Delta_i\right]$ satisfy Equation (\ref{eq:out17}).  
\end{remark}

To finish the paper, we make the following remarks about the intermediate subgroups of $2$-braid groups of surfaces.

\begin{remark}
Let $M$ be a compact, connected surface without boundary. If $H$ is a subgroup of $B_2(M)$ such that $P_2 (M) \subsetneq H$, then $H = B_2(M)$. Indeed, for such a subgroup $H$, there exists $\gamma \in H \setminus P_2(M)$, so $\langle P_n(M), \gamma \rangle \subset H$. Let $\beta  \in B_2(M)$ be such that $\pi(\beta) = (1 \ 2)$. So $\pi(\beta) = \pi(\gamma)$, and thus $\langle P_n(M), \beta \rangle = \langle P_n(M), \gamma \rangle \subset H$. Therefore, $\beta \in H$ for any $\beta \in B_2(M)$, which implies that $H = B_2(M)$.
\end{remark}

\begin{remark}
By Theorem \ref{theo:presentationBn}, we have $B_2(D) \cong \mathbb{Z}$. Therefore, $B_2(D)$ is bi-orderable. The groups $B_2(S^2)$ and $B_2(\mathbb{R}P^2)$ are finite \cite{VB}, so they are not left-orderable. For the torus $T$, $B_2(T)$ is left-orderable \cite[Theorem 4.2.1]{MP}, whereas it is not bi-orderable \cite{BGG}. However, in general, it is not known if $2$-braid groups of compact, connected surfaces without boundary are left-orderable. For the Klein bottle $K$, $B_2(K)$ cannot be bi-ordered \cite{GM}.
\end{remark}

{\it Acknowledgment:} This work is part of the author's doctoral thesis. I would like to thank Prof. Daciberg Lima Gonçalves and Prof. John Guaschi for discussions.  The author received support from the Brazilian agencies CAPES (Grant \#88887.639802/2021-00) and CNPq (Grant \#200666/2022-3).

\bibliographystyle{crplain}
\bibliography{samplebib}

\rule{4.5cm} {1pt} \vspace{2mm}

 {\sc Raquel Magalhães de Almeida Cruz} ({\sl raquelcruz@ime.usp.br})
 
 Universidade de S\~ao Paulo -- Instituto de Matem\'atica e Estat\'istica.
 
  Rua do Mat\~ao, 1010, Cidade Universit\'aria, 05508-090, S\~ao Paulo SP, Brasil.

 Normandie Univ., UNICAEN, CNRS, LMNO, 14000 Caen, France.
 
\vspace{3mm} $$\star \ \star \ \star$$

\end{document}